\documentclass[11pt]{article}
\usepackage[T1]{fontenc}
\usepackage[a4paper,margin=1.1in]{geometry}
\usepackage{amsmath,amssymb,amsthm,mathtools}
\usepackage{enumitem}
\usepackage{float}
\usepackage{microtype}
\usepackage[sort]{cite}
\usepackage{authblk}
\usepackage{xcolor}
\usepackage{hyperref}
\definecolor{mycyanblue}{HTML}{0080AC}
\hypersetup{
    colorlinks=true,
    allcolors=mycyanblue
}

\usepackage[nameinlink,capitalize]{cleveref}
\usepackage{tikz}
\usepackage{graphicx}
\usetikzlibrary{arrows.meta,calc,fit,positioning}

\newtheorem{theorem}{Theorem}[section]
\newtheorem{lemma}[theorem]{Lemma}
\newtheorem{proposition}[theorem]{Proposition}
\newtheorem{corollary}[theorem]{Corollary}

\theoremstyle{definition}

\theoremstyle{remark}
\newtheorem{remark}[theorem]{Remark}

\newcommand{\SpecC}{\mathcal S}
\newcommand{\CharC}{\mathcal C}
\newcommand{\dist}{\operatorname{dist}}
\newcommand{\rhoA}{\rho_A}

\title{The Distance Between the Adjacency Spectral Center\\ and the Characteristic Set of a Tree}

\author{
Chaochao Zhu\textsuperscript{a *},\quad
Shipei Hu\textsuperscript{a},\quad
Jingfu Huang\textsuperscript{a},\quad
Qin Yue\textsuperscript{a}\\
{\small
\textsuperscript{a}College of Finance and Mathematics, West Anhui University, Lu'an 237012, China}}
\date{}

\begin{document}
\maketitle

\begingroup
\renewcommand{\thefootnote}{}
\footnotetext{
*\,Corresponding author. E-mail: zccjsbz@amss.ac.cn\\
\texttt{081018024@fudan.edu.cn(S.Hu), huangjingfu@wxc.edu.cn(J.Huang), yqer@163.com(Q.Yue)} 
}
\endgroup

\begin{abstract}
Let $\SpecC(T)$ denote the adjacency spectral center of a tree $T$, and let $\CharC(T)$ be its characteristic set. We determine the maximum possible distance
\[
\Delta_n:=\max_{|V(T)|=n}\operatorname{dist}_T(\SpecC(T),\CharC(T))
\]
for every \(n\ge 3\). It turns out that \(\Delta_n=0\) for \(3\le n\le 11\), that \(\Delta_{12}=1\), and that
\[
\Delta_n=\left\lfloor\frac{n-11}{2}\right\rfloor
\quad(n\ge 13).
\]
More generally, if the two sets are at distance \(d\ge 2\), then the branch issuing from a closest vertex of $\SpecC(T)$ in the direction of $\CharC(T)$ contains at least \(2d+5\) vertices, while the complementary part of the tree contains at least $6$. We also establish structural restrictions in the equality case and completely classify the extremal trees of every odd order \(n\ge 17\): there are exactly four nonisomorphic extremal trees, consisting of two brooms and two double brooms. Our proofs exploit the opposite extremal behaviour of paths for the adjacency spectral radius and for the bottleneck Perron value.

\end{abstract}

\vspace{0.8em}


\vspace{0.5em}


\section{Introduction}\label{sec:introduction}

Several notions of centrality for trees need not coincide. The ordinary center is metric, the centroid is controlled by component sizes, and the subtree core by subtree counts. We compare two centers of spectral origin.

Let $T$ be a tree. The first center comes from the second adjacency eigenvalue $\lambda_2(T)$. Neumaier's theory, in the formulation of Brouwer and Haemers, yields a unique minimal subtree $Y$ such that every component of $T-Y$ has adjacency spectral radius at most $\lambda_2(T)$. The subtree $Y$ is always a vertex or an edge; we write it as $\SpecC(T)$ and call it the adjacency spectral center \cite{Neumaier1982,BrouwerHaemers2012}. The same object has recently reappeared in extremal questions for the second adjacency eigenvalue of trees~\cite{KumarMoharPragadaZhan2025}.

The second center is the characteristic set $\CharC(T)$ attached to the algebraic connectivity $\mu_2(T)$. On a tree, the sign pattern of a Fiedler vector picks out a vertex or an edge; see Fiedler~\cite{Fiedler1975} and Bapat~\cite{Bapat2014}. The resulting characteristic set is independent of the choice of the Fiedler vector
\cite{Merris1987}. Kirkland, Neumann, and Shader gave the equivalent Perron-branch description used below~\cite{KirklandNeumannShader1996}; see also \cite{AndradeDahl2017,AndradeCiardoDahl2022}.

Distances between central parts of trees have been studied for the center, centroid, characteristic set, subtree core, and related objects \cite{Patra2007,AbreuFritscherJustelKirkland2017,SmithSzekelyWangYuan2018,PandeyPatra2022}. Ciardo's framework of centers induced by rooted-subtree weights places the characteristic set in a broader setting \cite{Ciardo2020}. To the best of our knowledge, however, no sharp comparison has been obtained between the adjacency spectral center and the characteristic set. This pair is particularly natural because the two centers are governed by local branch weights with opposite extremal behavior: adjacency spectral radius favors branching, whereas bottleneck Perron value favors long rooted paths.

Our aim is to determine how far apart these two spectral centers can be, and to understand the structure of those trees that attain the maximum separation.

For nonempty vertex sets \(X,Y\subseteq V(T)\), write
\[
\operatorname{dist}_T(X,Y):=\min\{d_T(x,y):x\in X,\ y\in Y\},
\]
and define \(d(T):=\operatorname{dist}_T(\SpecC(T),\CharC(T)), \Delta_n:=\max_{|V(T)|=n}d(T),\) where the maximum is taken over all trees of order $n$.

For integers \(k,\ell\ge 1\), let \(B(k,\ell)\) denote the broom obtained from a path of length $k$ by adjoining $\ell$ additional leaves at one endpoint, so that
\[
|V(B(k,\ell))|=k+\ell+1.\]

Our first main result gives the exact maximum separation for every order.

\begin{theorem}\label{thm:main}
For every integer \(n\ge 3\),
\[
\Delta_n=
\begin{cases}
0, & 3\le n\le 11,\\[2mm]
1, & n=12,\\[2mm]
\displaystyle
\left\lfloor\frac{n-11}{2}\right\rfloor,
& n\ge 13.
\end{cases}
\]
Moreover, the four-leaf broom \(B(n-5,4)\) is extremal for every \(n\ge 12\).
\end{theorem}

The upper bound in Theorem~\ref{thm:main} is a consequence of a stronger structural result, which provides separate lower bounds for the two parts of a tree determined by a closest pair of vertices from the two spectral centers.

\begin{theorem}\label{thm:intro-separation}
Let $T$ be a tree, and set \(d:=\operatorname{dist}_T\bigl(\SpecC(T),\CharC(T)\bigr)\ge 2.\) Choose \(s\in \SpecC(T)\) and \(c\in \CharC(T)\) such that \(d_T(s,c)=d.\) Let $B$ be the component of $T-s$ containing $c$, and set \(H:=T\bigl[V(T)\setminus V(B)\bigr].\) Then $|V(B)|\ge 2d+5$ and $|V(H)|\ge 6$. In particular, \(|V(T)|\ge 2d+11.\)
\end{theorem}

The equality case of Theorem~\ref{thm:intro-separation} exhibits additional rigidity. In particular, the part of the tree opposite the characteristic set contains a five-vertex star, the path joining a closest pair of spectral-center and characteristic vertices has no branching vertices, and the characteristic set contains the unique centroid of the tree.

These structural restrictions allow us to go beyond the extremal value and to classify all extremal trees of sufficiently large odd order.

For integers \(k,a,b\ge 1\), let \(D(k;a,b)\) denote the double broom obtained from a path of length $k$ by adjoining $a$ leaves at one endpoint and $b$ leaves at the other, so that
\[
|V(D(k;a,b))|=k+a+b+1.\]

\begin{theorem}[Odd-order extremal trees]\label{thm:odd-extremal}
Let \(n=2d+11\), where \(d\ge 3\). An $n$-vertex tree $T$ satisfies \(d(T)=\Delta_n=d\) if and only if $T$ is isomorphic to one of
\(B(2d+6,4), B(2d+7,3), D(2d+4;2,4), D(2d+5;2,3).\) In particular, there are exactly four nonisomorphic extremal trees of every odd order \(n\ge 17\).
\end{theorem}

The proofs rest on two complementary extremal properties of paths. Among connected graphs of fixed order, paths minimize the adjacency spectral radius; among rooted trees of fixed order, endpoint-rooted paths maximize the bottleneck Perron value. These properties lead to two branch inequalities along a shortest path joining $\SpecC(T)$ and 
$\CharC(T)$. They also yield a one-sided path replacement, which we call pathification, that does not decrease the separation.

The branch inequalities establish the sharp bounds in Theorem~\ref{thm:intro-separation}. In the equality case, they further imply that a certain branch has adjacency spectral radius at most two. Combining this observation with Smith's classification of graphs with adjacency spectral radius at most two yields the structural characterization in Theorem~\ref{thm:odd-extremal}.

\section{Preliminaries}\label{sec:preliminaries}
Throughout, all graphs are finite, simple, and undirected. For a graph $G$, write $A(G)$ and $L(G)$ for its adjacency and Laplacian matrices. If $G$ is connected, its adjacency eigenvalues are ordered as \(\lambda_1(G)\ge\lambda_2(G)\ge\cdots\ge\lambda_n(G),\) and its Laplacian eigenvalues as \(0=\mu_1(G)<\mu_2(G)\le\cdots\le\mu_n(G).\)
For a graph $F$ that is not necessarily connected, put
\[
  \rhoA(F):=\max\{\rho(A(K)):K\text{ is a connected component of }F\}.
\]
For the empty graph we set $\rhoA(\varnothing):=0$. Thus, for every nonempty graph $F$, one has $\rhoA(F)=\rho(A(F))$.

\bigskip
\noindent\textit{The adjacency spectral center.}
\bigskip

Let $T$ be a tree with at least two vertices and let $\lambda=\lambda_2(T)$. The adjacency spectral center $\SpecC(T)$ is the unique minimal subtree $Y$ such that every
component of $T-Y$ has adjacency spectral radius at most $\lambda$. It is known that $\SpecC(T)$ consists of either one vertex or the two endpoints of one edge~\cite{BrouwerHaemers2012}.

For $v\in V(T)$, the components of $T-v$ are the \emph{branches at $v$}. The spectral center is detected locally by their radii.

\begin{lemma}[Adjacency branch criterion]\label{lem:adjacency-branch}
Let $B_1,\dots,B_k$ be the branches of a tree $T$ at a vertex $v$.
\begin{enumerate}[label=\textup{(\roman*)}]
\item If $\max_i\rho(A(B_i))$ is attained by at least two branches, then $\SpecC(T)=\{v\}$.
\item If $B_j$ is the unique branch of maximum adjacency spectral
radius, then
\[
\SpecC(T)\subseteq B_j\cup\{v\},
\quad
  \SpecC(T)\neq\{v\}.
\]
If $v\in\SpecC(T)$, then $\SpecC(T)=\{v,u\}$ where $u$ is the neighbor of $v$ in $B_j$.
\end{enumerate}
\end{lemma}

\begin{proof}
Let $r_1\ge r_2\ge\cdots$ be the eigenvalues of $A(T-v)$. If the maximum branch spectral radius is attained at least twice, with common value $r$, then
\(\lambda_1(T-v)=\lambda_2(T-v)=r.\) Cauchy interlacing gives
\[
\lambda_1(T-v)\ge\lambda_2(T)\ge\lambda_2(T-v),
\]
so $\lambda_2(T)=r$.  Hence every component of $T-v$ has spectral radius at most $\lambda_2(T)$, and the minimality of the spectral center gives $\SpecC(T)=\{v\}$.

Now suppose that $B_j$ is the unique branch of maximum spectral radius and put $r=\rho(A(B_j))$. Interlacing gives \(r=\lambda_1(T-v)\ge \lambda_2(T).\) We claim that the inequality is strict. Suppose instead that \(r=\lambda_2(T)\). By the equality case of Cauchy interlacing for the principal submatrix $A(T-v)$, there exists a $\lambda_2(T)$-eigenvector $x$ of $T$ with $x_v=0$ such that \(x|_{T-v}\) is an $r$-eigenvector of $T-v$. Since $B_j$ is the unique component of $T-v$ whose spectral radius is $r$, every other component has spectral radius strictly smaller than $r$. Hence \(x|_{T-v}\) is supported on $B_j$, and its restriction to $B_j$ is a nonzero multiple of the Perron vector of $B_j$. In particular, if $u$ is the vertex of $B_j$ adjacent to $v$, then \(x_u\ne0\). But the eigenvalue equation at $v$, together with \(x_v=0\), gives
\[
0=\lambda_2(T)x_v
=\sum_{w\sim v}x_w
  =x_u,
\]
because \(x\) vanishes on every other branch at \(v\). This contradiction shows that \(\rho(A(B_j))=r>\lambda_2(T).\) For every $i\ne j$, the eigenvalue $\rho(A(B_i))$ occurs below the largest eigenvalue of $A(T-v)$, so
\[
  \rho(A(B_i))\le\lambda_2(T-v)\le\lambda_2(T).
\]
Hence $B_j$ is the unique branch whose spectral radius exceeds $\lambda_2(T)$. Since the spectral center is a vertex or an edge, its defining property forces it to lie in $B_j\cup\{v\}$, and it cannot equal $\{v\}$. The final assertion follows from connectedness and the fact that $|\SpecC(T)|\le2$.
\end{proof}

\medskip
\noindent\textit{The characteristic set and Perron branches.}
\medskip

Let $R$ be a rooted tree with root $r$. For $u\in V(R)$, let $\mathcal P_r(u)$ be the root--$u$ path. Define the \emph{bottleneck matrix}
\[
M(R)_{uv}
  :=\bigl|V(\mathcal P_r(u))\cap V(\mathcal P_r(v))\bigr|,
\]
and its Perron value \(\beta(R):=\rho(M(R)).\) With our convention of counting common path vertices,
\[
  M(R)=\bigl(L(R)+e_re_r^{\mathsf T}\bigr)^{-1}.
\]
This is the grounded form of the usual inverse formula for a reduced incidence matrix of a tree; see \cite{Bapat2014,AndradeCiardoDahl2022}.

If $v$ is a vertex of an unrooted tree $T$, each branch of $T-v$ is rooted at its vertex adjacent to $v$. A branch of maximum Perron value is called a \emph{Perron branch at $v$}. The characteristic set has the following standard description \cite{KirklandNeumannShader1996,AndradeCiardoDahl2022}.

\begin{theorem}[Perron-branch orientation]\label{thm:perron-orientation}
Let $T$ be a tree and $v\in V(T)$.
\begin{enumerate}[label=\textup{(\roman*)}]
\item If at least two branches at $v$ have maximum Perron value, then $\CharC(T)=\{v\}$.
\item If the maximum Perron value is attained by a unique branch $B$, then
\[
\CharC(T)\subseteq B\cup\{v\},
  \quad
  \CharC(T)\neq\{v\}.
\]
If $v\in\CharC(T)$, then $\CharC(T)=\{v,u\}$, where $u$ is the neighbor of $v$ in $B$.
\item If $\CharC(T)=\{u,v\}$ is a characteristic edge, then the branch at $u$ containing $v$ and the branch at $v$ containing $u$ are the unique Perron branches at $u$ and $v$, respectively.
\end{enumerate}
\end{theorem}

\begin{remark}[Two oriented center mechanisms]
The two preceding branch criteria have the same formal shape. At a vertex $v$, orient toward the unique branch of maximum weight whenever such a branch exists. For adjacency spectral radius, a tie produces the spectral vertex; for bottleneck Perron value, a tie produces the characteristic vertex. If there is no tie, following the orientations leads to the corresponding center edge. For paths, however, the two branch weights have opposite extremal behavior.
\end{remark}

\begin{lemma}[Rooted-subtree monotonicity]\label{lem:rooted-subtree}
If $R$ is a rooted subtree of a rooted tree $R'$ with the same root, then \(\beta(R)\le\beta(R').\) The inequality is strict if the inclusion is proper.
\end{lemma}

\begin{proof}
The matrix $M(R)$ is a principal submatrix of $M(R')$. Since all entries of $M(R')$ are positive, Perron--Frobenius monotonicity gives the result, with strict inequality for a proper inclusion.
\end{proof}

Whenever $P_m$ is regarded as a rooted tree, its root is an endpoint. Set
\[
p_m:=\beta(P_m)
  =\frac{1}{2\left(1-\cos\frac{\pi}{2m+1}\right)}
  \qquad(m\ge1),
\]
and put $p_0=0$.  Indeed,
\[
  M(P_m)^{-1}=L(P_m)+e_1e_1^{\mathsf T},
\]
and a standard sine-vector calculation gives
\[
  \lambda_{\min}\!\left(L(P_m)+e_1e_1^{\mathsf T}\right)
  =2-2\cos\frac{\pi}{2m+1},
\]
which yields the displayed formula for $p_m$.  The sequence $(p_m)$ is strictly increasing.

\begin{lemma}[Rooted paths maximize the Perron value]\label{lem:path-max-perron}
Let $R$ be a rooted tree of order $m$. Then \(\beta(R)\le p_m,\) with equality if and only if $R$ is a path rooted at an endpoint.
\end{lemma}

\begin{proof}
Order the vertices $v_1,\dots,v_m$ so that every ancestor precedes each of its descendants. Then the root path to $v_i$ contains at most $i$ vertices, and hence
\(
  M(R)_{ij}\le\min\{i,j\}.
\)
Therefore, entrywise,
\[
0<M(R)\le K_m,
  \quad
  K_m:=\bigl(\min\{i,j\}\bigr)_{i,j=1}^m.
\]
The matrix $K_m$ is the bottleneck matrix of the endpoint-rooted path $P_m$. Perron--Frobenius monotonicity gives $\beta(R)\le p_m$. Equality forces $M(R)=K_m$, and then the diagonal entries show that the vertices have depths $0,1,\dots,m-1$; hence the ancestor relation is a single chain. Thus $R$ is an endpoint-rooted path.
\end{proof}

We also use the following theorem of Collatz and Sinogowitz~\cite{CollatzSinogowitz1957}.

\begin{theorem}[Collatz--Sinogowitz]\label{thm:path-min-adj}
If $G$ is a connected graph of order $m$, then
\[
\rho(A(G))\ge 2\cos\frac{\pi}{m+1},\]
with equality if and only if $G\cong P_m$.
\end{theorem}

\section{Pathification}\label{sec:pathification}

Let $T$ be a tree with \(d=\dist_T\bigl(\SpecC(T),\CharC(T)\bigr)>0.\) Choose $s\in\SpecC(T)$ and $c\in\CharC(T)$ such that $d_T(s,c)=d$, and write the $s$--$c$ path as
\[s=x_0,x_1,\dots,x_d=c.\] Let $B$ be the component of $T-s$ containing $c$, put $m:=|V(B)|$, and let \(H:=T[V(T)\setminus V(B)],\) rooted at $s$.

For $j\ge1$, define $H[j]$ by prepending $j-1$ new vertices to the root of $H$, with the new outer endpoint as root. Thus $H[1]=H$ and
\[|V(H[j])|=|V(H)|+j-1.\]

The \emph{pathification} of $T$ at $(s,c)$ is the tree $\widehat T$ obtained by deleting $B$ and attaching to $s$ a new path \(s-y_1-y_2-\cdots-y_m.\) The order of the tree is unchanged.

\begin{figure}[H]
\centering

\begin{tikzpicture}[
  scale=0.90,
  transform shape,
  v/.style={circle,fill=black,inner sep=1.5pt},
  box/.style={
    draw,
    rounded corners=2pt,
    minimum width=1.4cm,
    minimum height=0.7cm
  },
  every node/.style={font=\small}
]

\node[box] (H) at (0,0) {$H-s$};

\node[v] (s) at (1.8,0) {};
\node at (1.8,-0.38) {$s=x_0$};

\draw (H.east)--(s);

\node[v] (x1) at (3.0,0) {};
\node at (3.0,-0.38) {$x_1$};

\node at (4.0,0) {$\cdots$};

\node[v] (xdm) at (5.0,0) {};
\node at (5.0,-0.38) {$x_{d-1}$};

\node[v] (c) at (6.3,0) {};
\node at (6.3,-0.38) {$c=x_d$};

\draw (s)--(x1);
\draw (x1)--(3.55,0);
\draw (4.45,0)--(xdm);
\draw (xdm)--(c);

\node[v] (q1) at (5.0,0.90) {};
\node[v] (q2) at (5.60,1.35) {};
\draw (xdm)--(q1)--(q2);

\node[box] (F) at (8.0,0) {$F$};
\draw (c)--(F.west);

\node at (5.1,1.80) {$B,\qquad |V(B)|=m$};

\node at (4.0,-1.10) {$(a)\qquad T$};

\end{tikzpicture}

\vspace{0.25cm}

\noindent\makebox[\linewidth][c]{
  $\Downarrow\qquad\text{\emph{pathification}}$}

\vspace{0.10cm}

\begin{tikzpicture}[
   scale=0.90,
  transform shape,
  v/.style={circle,fill=black,inner sep=1.5pt},
  box/.style={
    draw,
    rounded corners=2pt,
    minimum width=1.4cm,
    minimum height=0.7cm
  },
  every node/.style={font=\small}
]

\node[box] (H2) at (0,0) {$H-s$};

\node[v] (s2) at (1.8,0) {};
\node at (1.8,-0.38) {$s$};

\draw (H2.east)--(s2);

\node[v] (y1) at (3.0,0) {};
\node at (3.0,-0.38) {$y_1$};

\node at (4.0,0) {$\cdots$};

\node[v] (yd) at (5.0,0) {};
\node at (5.0,0.42) {$y_d$};

\node at (6.1,0) {$\cdots$};

\node[v] (ym) at (7.4,0) {};
\node at (7.4,-0.38) {$y_m$};

\draw (s2)--(y1);
\draw (y1)--(3.55,0);
\draw (4.45,0)--(yd);
\draw (yd)--(5.55,0);
\draw (6.65,0)--(ym);

\node at (5.15,0.95) {new branch $P_m$};

\node[align=center] at (3.55,-1.15)
  {backward branch\\[-1pt]$H[d]$};

\node[align=center] at (6.65,-1.15)
  {forward branch\\[-1pt]$P_{m-d}$};

\node at (4.7,-2.00) {$(b)\qquad \widehat T$};

\end{tikzpicture}

\caption{
Schematic pathification at a closest pair $s\in\SpecC(T)$ and $c\in\CharC(T)$. The component $B$ of $T-s$ containing $c$ is replaced by an endpoint-rooted path $P_m$ of the same order, while the $H$-side is left unchanged. At $y_d$, the branch toward $s$ is $H[d]$, whereas the branch away from $s$ is $P_{m-d}$.}
\label{fig:pathification}
\end{figure}
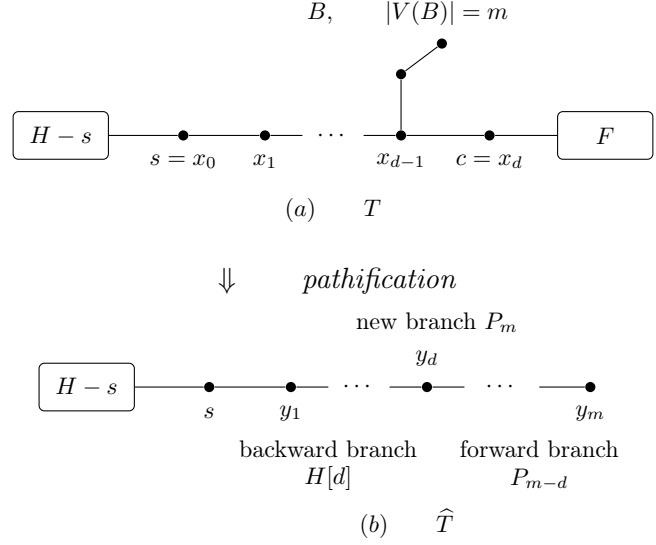

\begin{lemma}[Two branch inequalities]\label{lem:two-branch}
With the notation above,
\[
\rhoA(H-s)\ge 2\cos\frac{\pi}{m+1},
\]
and \(\beta(H[d])\le p_{m-d}.\) In particular, $m\ge d+1$.
\end{lemma}

\begin{proof}
Consider first the branches at $s$. If $\SpecC(T)=\{s\}$, then by Lemma~\ref{lem:adjacency-branch} the maximum adjacency branch weight at $s$ is attained at least twice. Hence some branch contained in $H-s$ has adjacency spectral radius at least $\rho(A(B))$.

If $\SpecC(T)$ is an edge and $s$ is its endpoint nearest to $c$, then the other endpoint lies in $H-s$. The branch containing that endpoint is the unique maximum adjacency branch at $s$, and therefore has spectral radius strictly larger than $\rho(A(B))$. Thus in all cases \(\rhoA(H-s)\ge\rho(A(B)).\) By Theorem~\ref{thm:path-min-adj},
\[
  \rho(A(B))\ge 2\cos\frac{\pi}{m+1},
\]
This proves the first inequality.

Now let $G_-$ be the branch at $c$ containing $s$. The rooted tree $H[d]$ is a rooted subtree of $G_-$: it contains the stem $x_{d-1},\dots,x_1,s$ together with all of $H$, while $G_-$ may have additional subtrees attached to the internal vertices of the $s$--$c$ path. Hence \(\beta(H[d])\le\beta(G_-)\) by Lemma~\ref{lem:rooted-subtree}.

If $\CharC(T)=\{c\}$, then there is a branch $F\ne G_-$ at $c$ with $\beta(F)\ge\beta(G_-)$. If $\CharC(T)$ is an edge and $c$ is its endpoint nearest to $s$, then the branch containing the other endpoint of the characteristic edge is the unique Perron branch at $c$, so again there is a branch $F\ne G_-$ with $\beta(F)>\beta(G_-)$. In either case,
\(\beta(H[d])\le\beta(F).\) The branch $F$ lies in $B\setminus\{x_1,\dots,x_d\}$, and therefore $|V(F)|\le m-d$. By Lemma~\ref{lem:path-max-perron},
\[
  \beta(F)\le p_{|V(F)|}\le p_{m-d}.
\]
This proves the second inequality.  Since $\beta(H[d])>0$, necessarily $m-d\ge1$.
\end{proof}

\begin{theorem}[Pathification]\label{thm:pathification}
The pathification $\widehat T$ satisfies
\[
  \dist_{\widehat T}\bigl(\SpecC(\widehat T),\CharC(\widehat T)\bigr)
  \ge
  \dist_T\bigl(\SpecC(T),\CharC(T)\bigr).
\]
\end{theorem}

\begin{proof}
By Lemma~\ref{lem:two-branch}, \(\rho_A(H-s)\ge \rho(A(P_m)).\) Hence the new path branch at $s$ is not the unique branch of maximum adjacency weight. If it ties for the maximum, Lemma~\ref{lem:adjacency-branch} gives \(\mathcal S(\widehat T)=\{s\}\); otherwise every maximum branch lies on the $H$-side. In either case, \(\mathcal S(\widehat T)\subseteq V(H).\) At $y_d$, the branch toward $s$ is precisely $H[d]$, whereas the forward branch is the endpoint-rooted path $P_{m-d}$. By Lemma~\ref{lem:two-branch},
\[
\beta(H[d])\le p_{m-d}=\beta(P_{m-d}).
\]
If equality holds, the two branches at $y_d$ are both Perron branches, and Theorem~\ref{thm:perron-orientation} gives \(\mathcal C(\widehat T)=\{y_d\}\). If the inequality is strict, the forward branch is the unique Perron branch at $y_d$, so Theorem~\ref{thm:perron-orientation} places the characteristic set on the forward side of $y_d$. Thus in either case
\[
\mathcal C(\widehat T)
   \subseteq \{y_d,y_{d+1},\ldots,y_m\}.
\]
Every path from a vertex of $H$ to \(\{y_d,y_{d+1},\ldots,y_m\}\) passes through $s$, and \(d_{\widehat T}(s,y_d)=d.\) Together with the two inclusions above, this yields
\[
\operatorname{dist}_{\widehat T}
\bigl(\mathcal S(\widehat T),\mathcal C(\widehat T)\bigr)
   \ge d =
\operatorname{dist}_{T}
   \bigl(\mathcal S(T),\mathcal C(T)\bigr).
\]
\end{proof}

\begin{remark}
The one-sided nature of pathification is essential. The replacement path is no heavier for adjacency than the original branch, whereas among rooted trees of the same order it has the largest bottleneck Perron value.
\par
Let $T$ be an extremal $n$-vertex tree with \(d(T)=\Delta_n>0\), and let \(\widehat T\) be its pathification at a closest pair \((s,c)\). By Theorem~\ref{thm:pathification}, \(\widehat T\) is again extremal, so
\[
d(\widehat T)=d(T)=d.
\]
Moreover, the proof of Theorem~\ref{thm:pathification} gives
\[
\mathcal{S}(\widehat T)\subseteq V(H),
\quad
\mathcal{C}(\widehat T)\subseteq \{y_d,y_{d+1},\ldots,y_m\}.
\]
Since every path joining these two sets passes through $s$ and $y_d$, the equality \(d(\widehat T)=d\) forces \(s\in \mathcal{S}(\widehat T), y_d\in \mathcal{C}(\widehat T).\) Thus, whenever \(\Delta_n>0\), an extremal $n$-vertex tree may be chosen so that the branch at a closest vertex of $\mathcal{S}(T)$ containing $\mathcal{C}(T)$ is an endpoint-rooted path.
\end{remark}

\section{The separation theorem}\label{sec:separation}
Lemma~\ref{lem:two-branch} reduces the sharp bound to comparing the Perron value of $H[d]$ with that of rooted brooms of the same stem length.

\medskip
\noindent\textit{Two rooted-broom estimates}
\medskip

For integers $a,\ell\ge1$, let $W(a,\ell)$ be the rooted tree obtained from the path
\[
u_0-u_1-\cdots-u_{a-1},\]
rooted at $u_0$, by adjoining $\ell$ leaves to $u_{a-1}$.

\begin{lemma}[Outward grafting]\label{lem:outward-grafting-short}
Let $R$ be a rooted tree. Suppose a nonempty pendant subtree $X$, not containing the root, is attached to a vertex $u\notin V(X)$. Let $v$ be a proper descendant of $u$ that does not belong to $X$. If $X$ is detached from $u$ and reattached at $v$, with the root unchanged, then the Perron value strictly increases.
\end{lemma}

\begin{proof}
After identifying the vertex sets, every pair of root paths has a common initial segment at least as long as before, while the root paths to vertices of $X$ become strictly longer. Hence the new bottleneck matrix dominates the old one entrywise and differs from it in at least one entry. Strict Perron--Frobenius monotonicity gives the claim.
\end{proof}

\begin{lemma}\label{lem:broom-benchmarks}
The following hold:
\[
\beta(W(a,4))>p_{a+3}\quad(a\ge3),
\]
and
\[
\beta(W(a,5))>p_{a+3}\quad(a\ge2).
\]
\end{lemma}

\begin{proof}
Fix $a\ge2$ and set
\[
\theta:=\frac{\pi}{2a+7},
  \quad
  \sigma:=2-2\cos\theta
 =\frac1{p_{a+3}},
  \quad
  z:=2\cos\theta.
\]
Let $R=W(a,\ell)$ and \(\mathcal L(R):=L(R)+e_{u_0}e_{u_0}^{\mathsf T}.\) Since $M(R)=\mathcal L(R)^{-1}$,
\[
  \beta(R)=\frac1{\lambda_{\min}(\mathcal L(R))}.
\]

On the stem put
\[
x_j:=\frac{\sin((j+1)\theta)}{\sin\theta},
  \quad 0\le j\le a-1,
\]
and assign to each terminal leaf the value \[y:=\frac{x_{a-1}}{1-\sigma}.\] Since $a\ge2$, we have $\theta\le\pi/11$, and hence \(1-\sigma=2\cos\theta-1>0.\) Also $0<(j+1)\theta<\pi/2$ for $0\le j\le a-1$, so all coordinates of the test vector are positive. The recurrence $x_{j+1}=zx_j-x_{j-1}$ shows that $(\mathcal L(R)-\sigma I)x$ vanishes at every coordinate except possibly $u_{a-1}$. If $\mathcal R$ denotes the remaining residual, then
\[
\frac{\mathcal R}{x_{a-1}}
  =1-\sigma-\frac{\ell\sigma}{1-\sigma}
   -\frac{\sin((a-1)\theta)}{\sin(a\theta)}.
\]
Let $\tau=\lambda_{\min}(\mathcal L(R))$. Since $\mathcal L(R)$ is an irreducible symmetric nonsingular $M$-matrix, a $\tau$-eigenvector $w$ may be chosen strictly positive.  Taking the inner product with the displayed test vector gives
\[
(\tau-\sigma)w^{\mathsf T}x
=w^{\mathsf T}(\mathcal L(R)-\sigma I)x
=w_{a-1}\mathcal R.
\]
Since $w^{\mathsf T}x>0$ and $w_{a-1}>0$,
\[
\operatorname{sgn}(\tau-\sigma)=\operatorname{sgn}(\mathcal R).
\]
Using $\beta(R)=1/\tau$ and $p_{a+3}=1/\sigma$, we obtain
\[
\beta(W(a,\ell))>p_{a+3}
  \quad\Longleftrightarrow\quad
  \tau<\sigma
  \quad\Longleftrightarrow\quad
  \mathcal R<0.
\]
Since \((2a+7)\theta=\pi\), \(a\theta=\frac{\pi}{2}-\frac{7\theta}{2}, (a-1)\theta=\frac{\pi}{2}-\frac{9\theta}{2},\) and hence
\[
\frac{\sin((a-1)\theta)}{\sin(a\theta)}=
\frac{\cos(9\theta/2)}{\cos(7\theta/2)}.
\]
Writing \(z=2\cos\theta\), the standard multiple-angle identities give
\[
\frac{\cos(9\theta/2)}{\cos(7\theta/2)}=
\frac{(z-1)(z^3-3z-1)}{z^3-z^2-2z+1}.
\]
Since \(\sigma=2-z, 1-\sigma=z-1,\) the residual therefore satisfies
\[
\frac{\mathcal R}{x_{a-1}}=(z-1)-\frac{\ell(2-z)}{z-1}-\frac{(z-1)(z^3-3z-1)}{z^3-z^2-2z+1}  \\
=\frac{2-z}{z-1}\left(\frac{(z-1)^2(z+1)}{z^3-z^2-2z+1}-\ell\right).
\]
Moreover,
\[\frac{(z-1)^2(z+1)}{z^3-z^2-2z+1}=
1+\frac{z}{z^3-z^2-2z+1}.
\]
Since \(1<z<2\), the prefactor \((2-z)/(z-1)\) is positive. Consequently,
\[
\beta(W(a,\ell))>p_{a+3}
\quad\Longleftrightarrow\quad
\mathcal R<0
\quad\Longleftrightarrow\quad
\ell>G_3(z),
\]
where
\[
G_3(z):=1+\frac{z}{z^3-z^2-2z+1}.
\]
Put $D(z):=z^3-z^2-2z+1$. At the left endpoint,
\[
D\!\left(\frac{191}{100}\right)
  =\frac{499771}{10^6}>0,
\]
and
\[
D'(z)=3z^2-2z-2>0
  \quad\left(z\ge\frac{191}{100}\right).
\]
Thus $D(z)>0$ throughout the interval used below. Moreover,
\[
G_3'(z)
  =-
  \frac{(z-1)(2z^2+z+1)}{(z^3-z^2-2z+1)^2}
  <0.
\]

If $a\ge2$, then $\theta\le\pi/11$. Using $2\cos x>2-x^2$ and $\pi<22/7$ gives $z>191/100$, and hence
\[
G_3(z)
  <G_3\left(\frac{191}{100}\right)
  =\frac{2409771}{499771}<5.
\]
Therefore $\beta(W(a,5))>p_{a+3}$ for $a\ge2$.

If $a\ge3$, then $\theta\le\pi/13$, so similarly $z>97/50$. Hence
\[
G_3(z)
 <G_3\left(\frac{97}{50}\right)
  =\frac{324723}{82223}<4,
\]
and therefore $\beta(W(a,4))>p_{a+3}$ for $a\ge3$.
\end{proof}

\begin{lemma}[Stem--broom domination]\label{lem:stem-broom-short}
Let $H$ be a rooted tree of order $h\ge2$ with root $s$, and let $d\ge1$. Then \(\beta(H[d])\ge\beta(W(d,h-1)).\) Equality holds if and only if every vertex of $H-s$ is adjacent to $s$.
\end{lemma}

\begin{proof}
Write the initial stem of $H[d]$ as
\[
  r_0-r_1-\cdots-r_{d-1}=s,
\]
where $r_0$ is the root, and let $w_1,\dots,w_{h-1}$ be the vertices of $H-s$. Compare $H[d]$ with $W(d,h-1)$ by identifying the stem vertices and the $h-1$ remaining vertices.

Entries of the two bottleneck matrices involving only stem vertices are equal. If one index is a stem vertex, the two root paths have the same common initial segment in both trees. Each root path to a nonstem vertex in $H[d]$ contains at least $d+1$ vertices, and every pair of such paths has at least the whole $d$-vertex stem in common. Thus
\(M(H[d])\ge M(W(d,h-1))\) entrywise, and Perron--Frobenius monotonicity gives the inequality. Equality holds exactly when every nonstem vertex is adjacent to $s$.
\end{proof}

\begin{lemma}[Six-vertex barrier]\label{lem:six-vertex-sides}
Let $T$ satisfy $d(T)>0$, and choose $s,c,B,H$ as in Section~\ref{sec:pathification}. Then
\(
|V(B)|\ge6,
  |V(H)|\ge6.
\)
In particular, disjoint center sets require at least twelve vertices.
\end{lemma}

\begin{proof}
Write $d=d(T)$, $m=|V(B)|$, and $h=|V(H)|$. On the vertex set $V(H)$, the corresponding principal submatrix of $M(H[d])$ is \(M(H)+(d-1)J_h,\) where $J_h$ is the all-ones matrix. Since \(M(H)+(d-1)J_h\ge M(H)\) entrywise, Perron--Frobenius monotonicity, followed by Cauchy interlacing for the principal submatrix of $M(H[d])$, gives
\[
\beta(H)
  \le
\rho\!\left(M(H)+(d-1)J_h\right)
 \le
  \beta(H[d])
  \le
p_{m-d}
  \le
p_{m-1}.
\]
Moreover, Lemma~\ref{lem:two-branch} gives
\begin{equation}\label{eq:six-vertex-constraints}
m\ge2,
  \quad
\rhoA(H-s)\ge2\cos\frac{\pi}{m+1}.
\end{equation}

First note that \(\beta(W(2,3))>p_4.\) With the root first, then the terminal stem vertex, and then the three leaves, the bottleneck matrix of $W(2,3)$ is
\[
M=
\begin{pmatrix}
  1&1&1&1&1\\
  1&2&2&2&2\\
  1&2&3&2&2\\
  1&2&2&3&2\\
  1&2&2&2&3
  \end{pmatrix}.
\]
For $x=(1,2,2,2,2)^{\mathsf T}$,
\[
\frac{x^{\mathsf T}Mx}{x^{\mathsf T}x}
  =\frac{157}{17}>9.
\]
Thus $\beta(W(2,3))>9$. On the other hand, strict concavity of sine on $[0,\pi/6]$ gives
\[
\sin\frac{\pi}{18}>
  \frac13\sin\frac{\pi}{6}=\frac16,
\]
and therefore
\[
 p_4=\frac{1}{4\sin^2(\pi/18)}<9.
\]
Consequently
\begin{equation}\label{eq:W23-p4}
\beta(W(2,3))>p_4.
\end{equation}

For every tree $K$ of order $q$,
\begin{equation}\label{eq:tree-radius-small}
\rho(A(K))\le\sqrt{q-1}.
\end{equation}  
Indeed, after ordering the two bipartition classes,
\[
A(K)=\begin{pmatrix}0&N\\ N^{\mathsf T}&0\end{pmatrix},
\]
so
\[
\rho(A(K))=\|N\|_2\le\|N\|_{\mathrm F}=\sqrt{|E(K)|}=\sqrt{q-1}.
\]

Suppose first that $m\le3$. Then $\rhoA(H-s)\ge1$ by \eqref{eq:six-vertex-constraints}; hence some component of $H-s$ contains an edge. Together with the root $s$, it contains an endpoint-rooted $P_3$. Thus
\[
\beta(H)\ge p_3>p_2\ge p_{m-1},\] a contradiction.

Now suppose $4\le m\le5$. Then \(\rhoA(H-s)\ge2\cos\frac{\pi}{5}>\sqrt2.\) By \eqref{eq:tree-radius-small}, some component $K$ of $H-s$ has order $q\ge4$. Root $K$ at its neighbor of $s$. The rooted tree $K[2]$ is a subtree of $H$, and Lemma~\ref{lem:rooted-subtree} and~\ref{lem:stem-broom-short} give
\[
\beta(H)
 \ge\beta(K[2])
  \ge\beta(W(2,q-1))
  \ge\beta(W(2,3))
 >p_4
  \ge p_{m-1},
\]
again impossible. Hence $m\ge6$.

Finally, \(\rhoA(H-s)\ge2\cos\frac{\pi}{7}>\sqrt3.\) Another application of \eqref{eq:tree-radius-small} shows that some component of $H-s$ has at least five vertices.  Adding $s$ yields $h\ge6$.
\end{proof}

\medskip
\noindent\textit{Sharp separation}
\medskip

\begin{proof}[Proof of Theorem~\ref{thm:intro-separation}]
Write \(m:=|V(B)|\) and \(h:=|V(H)|\). Lemma~\ref{lem:six-vertex-sides} gives \(m\ge 6\), so Lemma~\ref{lem:two-branch} yields
\[
\rho_A(H-s)\ge 2\cos\frac{\pi}{m+1} \ge 2\cos\frac{\pi}{7} >\sqrt{3}.
\]

Choose a component \(K\) of \(H-s\) with \(\rho(A(K))=\rho_A(H-s)\), and let \(q:=|V(K)|\). We shall use the standard bound, already recalled in the proof of Lemma~\ref{lem:six-vertex-sides},
that every tree of order \(q\) satisfies
\[
\rho(A(K))\le\sqrt{q-1}.
\]
Since \(\rho(A(K))>\sqrt{3}\), this forces \(q\ge 5\).

Now root \(K\) at its unique vertex adjacent to \(s\). Then \(K[d+1]\) is a rooted subtree of \(H[d]\) with the same root, and Lemmas~\ref{lem:rooted-subtree} and~\ref{lem:stem-broom-short} give
\[
\beta(H[d])
\ge \beta(K[d+1])
\ge \beta(W(d+1,q-1))
\ge \beta(W(d+1,4)).
\]
Because \(d\ge 2\), we have \(d+1\ge 3\), so Lemma~\ref{lem:broom-benchmarks} implies
\[
\beta(H[d])\ge \beta(W(d+1,4))>p_{d+4}.
\]
On the other hand, Lemma~\ref{lem:two-branch} gives $\beta(H[d])\le p_{m-d}$. Hence \(p_{m-d}>p_{d+4}.\) Since $(p_j)$ is strictly increasing, it follows that
\(m-d>d+4,\) and therefore \(m\ge 2d+5.\)

Finally, \(H\) contains \(s\) together with the component \(K\), so
\[
h\ge q+1\ge 6.
\]
Combining the two bounds gives \(|V(T)|=m+h\ge 2d+11,\) as required.
\end{proof}

\begin{corollary}\label{cor:separation-upper}
If $T$ is an $n$-vertex tree and $d(T)\ge2$, then
\[
d(T)\le\left\lfloor\frac{n-11}{2}\right\rfloor.\]
\end{corollary}

\begin{corollary}[Structure at equality]\label{cor:equality}
Let $T$ be a tree with $d:=d(T)\ge 2$ and $|V(T)|=2d+11$. Choose \(s\in S(T)\) and \(c\in C(T)\) such that \(d_T(s,c)=d\), and write the $s$--$c$ path as \(s=x_0,x_1,\ldots,x_d=c.\) Let $B$ and $H$ be defined as in Section~\ref{sec:pathification}. Then \(H-s\cong K_{1,4},\) and \(\deg_T(x_i)=2 \quad(0\le i\le d).\) Moreover, the two components of $T-c$ each have order $d+5$. In particular, $c$ is the unique centroid of $T$.
\end{corollary}

\begin{proof}
Put \(m:=|V(B)|\) and \(h:=|V(H)|\). By Theorem~\ref{thm:intro-separation}, \(m\ge 2d+5, h\ge 6\). Since \(m+h=|V(T)|=2d+11\), equality must hold in both inequalities; hence
\(m=2d+5, h=6.\) Now Lemma~\ref{lem:two-branch} gives
\[
\rho_A(H-s) \ge 2\cos\frac{\pi}{m+1}=2\cos\frac{\pi}{2d+6} \ge 2\cos\frac{\pi}{10} >\sqrt{2+\sqrt{2}}.
\]
Because \(h=6\), the forest \(H-s\) has five vertices. Among forests of order five, \(K_{1,4}\) is the only one whose adjacency spectral radius exceeds \(\sqrt{2+\sqrt{2}}\).
Therefore
\[
H-s\cong K_{1,4}.
\]

It remains to determine the structure along the $s$--$c$ path. Recall from the proof of Theorem~\ref{thm:intro-separation} that
\[
\beta(H[d])>p_{d+4}.
\]
As in the proof of Lemma~\ref{lem:two-branch}, there is a branch $F$ at $c$, not containing $s$, such that
\[
\beta(H[d])\le\beta(F)
\quad\text{and}\quad
|V(F)|\le m-d.
\]
By Lemma~\ref{lem:path-max-perron},
\[
p_{|V(F)|} \ge\beta(F) \ge\beta(H[d]) >p_{d+4}.
\]
Since $(p_j)$ is strictly increasing, we obtain \(|V(F)|\ge d+5.\) On the other hand,
\[
|V(F)|\le m-d=d+5.
\]
Hence \(|V(F)|=d+5.\)

Since $B$ has exactly $2d+5$ vertices, the vertices \(x_1,\ldots,x_d\) together with $V(F)$ exhaust $V(B)$. Consequently, no vertex of the \(s\)--\(c\) path has an additional branch inside $B$. Moreover, \(H-s\) is connected, so $s$ has exactly one neighbour in $H$.
It follows that
\[
\deg_T(x_i)=2 \quad(0\le i\le d).
\]

Finally, the two components of \(T-c\) have orders
\[
h+d-1=d+5
\quad\text{and}\quad
|V(F)|=d+5,
\]
respectively. Thus $c$ separates $T$ into two components of equal order. Since \(|V(T)|=2d+11\) is odd, $c$ is the unique centroid of $T$.
\end{proof}

\section{The exact extremal distance}\label{sec:exact-distance}

The four-leaf broom family attains the upper bounds, beginning with $B(7,4)$ at order $12$.

\medskip
\noindent\textit{A distance bound from the characteristic set}

\begin{lemma}\label{lem:char-distance-short}
Let $T$ be a tree of order $n$, let $x\in V(T)$, and let $c\in\CharC(T)$. Then
\[
d_T(x,c)\le\left\lfloor\frac n2\right\rfloor.
\]
\end{lemma}

\begin{proof}
Suppose first that $\CharC(T)=\{v\}$. Let $h=d_T(x,v)$. If $h=0$ there is nothing to prove. Otherwise, let $B$ be the branch at $v$ containing $x$. As a rooted tree, $B$
contains an endpoint-rooted $P_h$, so $\beta(B)\ge p_h$. Since $v$ is a characteristic vertex, there is another branch $B'$ at $v$ with $\beta(B')\ge\beta(B)$. If $m'=|V(B')|$, then \(p_{m'}\ge\beta(B')\ge p_h,\) so $m'\ge h$. Also $|V(B)|\ge h$, and hence
\[
n\ge1+|V(B)|+|V(B')|\ge2h+1.
\]
Thus $h\le\lfloor(n-1)/2\rfloor$.

Now suppose $\CharC(T)=\{u,v\}$ is an edge. Choose the endpoint $u$ nearer to $x$ and let
\[
h:=d_T(x,u)=\dist_T(x,\CharC(T)).
\]
If $h=0$ the claim is immediate. Let $B$ be the branch at $u$ containing $x$, and let $B'$ be the branch containing $v$. Then $\beta(B)\ge p_h$, while $B'$ is the unique Perron branch at $u$. Thus \(\beta(B')>p_h.\) If $m'=|V(B')|$, then $p_{m'}>p_h$, so $m'\ge h+1$. Consequently \(n\ge1+h+(h+1)=2h+2.\) Thus $h\le\lfloor(n-2)/2\rfloor$. The other endpoint of the characteristic edge is one step farther away, so every $c\in\CharC(T)$ satisfies
\[
d_T(x,c)\le\left\lfloor\frac{n-2}{2}\right\rfloor+1
  =\left\lfloor\frac n2\right\rfloor.
\]
\end{proof}

\medskip
\noindent\textit{The extremal broom}
\medskip

Recall the broom $B(k,\ell)$ introduced in the Introduction. We label its underlying path as
\[
v_0-v_1-\cdots-v_k,\]
with the $\ell$ additional leaves attached at $v_k$.

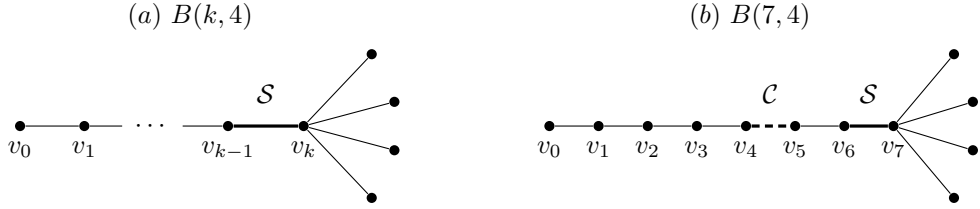
\begin{figure}[H]
\resizebox{\linewidth}{!}{%
\centering
\begin{tikzpicture}[
  v/.style={circle,fill=black,inner sep=1.35pt},
  every node/.style={font=\small}
]

\begin{scope}

\node[v,label=below:$v_0$] (a0) at (0,0) {};
\node[v,label=below:$v_1$] (a1) at (0.85,0) {};

\node at (1.75,0) {$\cdots$};

\node[v,label=below:$v_{k-1}$] (akm) at (2.75,0) {};
\node[v,label=below:$v_k$] (ak) at (3.75,0) {};

\draw (a0)--(a1);
\draw (a1)--(1.35,0);
\draw (2.15,0)--(akm);

\draw[very thick] (akm)--(ak);

\node[v] (al1) at (4.65,0.95) {};
\node[v] (al2) at (4.95,0.32) {};
\node[v] (al3) at (4.95,-0.32) {};
\node[v] (al4) at (4.65,-0.95) {};

\draw (ak)--(al1);
\draw (ak)--(al2);
\draw (ak)--(al3);
\draw (ak)--(al4);

\node[above=5pt] at ($(akm)!0.5!(ak)$) {$\SpecC$};
\node at (2.25,1.45) {$(a)\ B(k,4)$};

\end{scope}

\begin{scope}[xshift=7.0cm]

\node[v,label=below:$v_0$] (b0) at (0.00,0) {};
\node[v,label=below:$v_1$] (b1) at (0.65,0) {};
\node[v,label=below:$v_2$] (b2) at (1.30,0) {};
\node[v,label=below:$v_3$] (b3) at (1.95,0) {};
\node[v,label=below:$v_4$] (b4) at (2.60,0) {};
\node[v,label=below:$v_5$] (b5) at (3.25,0) {};
\node[v,label=below:$v_6$] (b6) at (3.90,0) {};
\node[v,label=below:$v_7$] (b7) at (4.55,0) {};

\draw (b0)--(b1);
\draw (b1)--(b2);
\draw (b2)--(b3);
\draw (b3)--(b4);

\draw[very thick,densely dashed] (b4)--(b5);

\draw (b5)--(b6);

\draw[very thick] (b6)--(b7);

\node[v] (bl1) at (5.35,0.95) {};
\node[v] (bl2) at (5.60,0.32) {};
\node[v] (bl3) at (5.60,-0.32) {};
\node[v] (bl4) at (5.35,-0.95) {};

\draw (b7)--(bl1);
\draw (b7)--(bl2);
\draw (b7)--(bl3);
\draw (b7)--(bl4);

\node[above=5pt] at ($(b4)!0.5!(b5)$) {$\CharC$};
\node[above=5pt] at ($(b6)!0.5!(b7)$) {$\SpecC$};

\node at (2.65,1.45) {$(b)\ B(7,4)$};

\end{scope}

\end{tikzpicture}
}
\caption{
The four-leaf broom. In $(a)$ the heavy edge is $\SpecC(B(k,4))=\{v_{k-1},v_k\}$. For the twelve-vertex broom in $(b)$, the dashed heavy edge is $\CharC(B(7,4))=\{v_4,v_5\}$,
whereas the solid heavy edge is $\SpecC(B(7,4))=\{v_6,v_7\}$.}
\label{fig:four-leaf-broom}
\end{figure}

\begin{lemma}\label{lem:broom-center-short}
For every $k\ge2$, \(\SpecC(B(k,4))=\{v_{k-1},v_k\}.\)
\end{lemma}

\begin{proof}
Let \(T=B(k,4)\), and put \(a=v_{k-1}\) and \(b=v_k\). Let $B_a$ be the branch at $a$ containing $b$, and let $B_b$ be the branch at $b$ containing $a$.
It follows that \(B_a\cong K_{1,4}\) and \(B_b\cong P_k\).

At $a$, the branch $B_a$ is the unique branch of maximum adjacency spectral radius, since \(\rho(A(K_{1,4}))=2>\rho(A(P_{k-1})).\) Similarly, at $b$, the branch $B_b$ is also the unique branch of maximum adjacency spectral radius, since \(k\ge2\) and all other branches are isolated vertices. By Lemma~\ref{lem:adjacency-branch},
\[\mathcal S(T)
\subseteq
\bigl(V(B_a)\cup\{a\}\bigr)
\cap
\bigl(V(B_b)\cup\{b\}\bigr)=\{a,b\}.\]
The same lemma excludes both \(\mathcal S(T)=\{a\}\) and \(\mathcal S(T)=\{b\}\). Hence \(\mathcal S(T)=\{a,b\}\).
\end{proof}

\begin{proposition}[Small orders]\label{prop:small-orders-analytic}
For $3\le n\le11$ one has $\Delta_n=0$, whereas $\Delta_{12}=1$. For the twelve-vertex broom $B(7,4)$,
\(
\SpecC(B(7,4))=\{v_6,v_7\},
  \quad
\CharC(B(7,4))=\{v_4,v_5\}.
\)
\end{proposition}

\begin{proof}
If $3\le n\le11$, Lemma~\ref{lem:six-vertex-sides} forbids disjoint center sets. Hence $d(T)=0$ for every tree of order $n$, and $\Delta_n=0$.

Now let $T=B(7,4)$. By Lemma~\ref{lem:broom-center-short},
\(
\SpecC(T)=\{v_6,v_7\}.\)
At $v_5$, the branch toward $v_6$ is $W(2,4)$, while the other branch is the endpoint-rooted path $P_5$. The row sums of $M(W(2,4))$ are $6,11,12,12,12,12$, so
\(\beta(W(2,4))\le12.\) Since $2-2\cos t<t^2$ for $t>0$ and $\pi<22/7$,
\[
2-2\cos\frac{\pi}{11}
<\left(\frac{\pi}{11}\right)^2
  <\frac4{49}.
\]
Thus
\[
  p_5>\frac{49}{4}>12\ge\beta(W(2,4)),
\]
and the unique Perron branch at $v_5$ points toward $v_4$.

At $v_4$, the branch toward $v_5$ is $W(3,4)$ and the other branch is $P_4$. Since $W(3,4)$ properly contains an endpoint-rooted $P_4$, Lemma~\ref{lem:rooted-subtree} gives
\(\beta(W(3,4))>p_4.\) Hence the unique Perron branch at $v_4$ points toward $v_5$, while, as shown above, the unique Perron branch at $v_5$ points toward $v_4$.
Let $B_4$ be the branch at $v_4$ containing $v_5$, and let $B_5$ be the branch at $v_5$ containing $v_4$. By Theorem~\ref{thm:perron-orientation},
\[
\CharC(T)
 \subseteq
\bigl(B_4\cup\{v_4\}\bigr)
 \cap
\bigl(B_5\cup\{v_5\}\bigr)=
\{v_4,v_5\}.
\]
Since the maximum Perron branch is unique at both $v_4$ and $v_5$, neither vertex can by itself be the characteristic set. Therefore \(\CharC(T)=\{v_4,v_5\}.\)
We get $d(T)=1$, so $\Delta_{12}\ge1$. On the other hand, Theorem~\ref{thm:intro-separation} rules out $d(T)\ge2$ for a twelve-vertex tree. Thus $\Delta_{12}=1$.
\end{proof}

\begin{proposition}\label{prop:broom-lower-short}
Let $n\ge13$ and put $T_n:=B(n-5,4)$. Then
\[
 \dist_{T_n}\bigl(\SpecC(T_n),\CharC(T_n)\bigr)
  \ge
  \left\lfloor\frac{n-11}{2}\right\rfloor.
\]
\end{proposition}

\begin{proof}
Put $k=n-5$. By Lemma~\ref{lem:broom-center-short}, \(\SpecC(T_n)=\{v_{k-1},v_k\}.\) Let $s\in\SpecC(T_n)$ and $c\in\CharC(T_n)$. Then
\[
  d_{T_n}(v_0,s)\ge k-1=n-6,
\]
whereas Lemma~\ref{lem:char-distance-short} gives \(d_{T_n}(v_0,c)\le\left\lfloor\frac n2\right\rfloor.\) By the triangle inequality,
\[
d_{T_n}(s,c)
  \ge
n-6-\left\lfloor\frac n2\right\rfloor=
\left\lfloor\frac{n-11}{2}\right\rfloor.
\]
Taking the minimum over $s$ and $c$ proves the result.
\end{proof}

\begin{proof}[Proof of Theorem~\ref{thm:main}]
The orders $3\le n\le12$ are settled by Proposition~\ref{prop:small-orders-analytic}. Assume henceforth that $n\ge13$. The lower bound follows from Proposition~\ref{prop:broom-lower-short}.

For the reverse inequality, let $T$ be any tree of order $n$ and put
\[
d:=\dist_T\bigl(\SpecC(T),\CharC(T)\bigr).
\]
If $n\in\{13,14\}$ and $d\ge2$, then Theorem~\ref{thm:intro-separation} gives $n\ge2d+11\ge15$, a contradiction. Hence $d\le1$, which is exactly $\lfloor(n-11)/2\rfloor$.

Now let $n\ge15$. If $d\le1$, there is nothing to prove. If $d\ge2$, Theorem~\ref{thm:intro-separation} gives $n\ge2d+11$, and therefore
\[
 d\le\left\lfloor\frac{n-11}{2}\right\rfloor.
\]
Thus the upper and lower bounds coincide. The broom $B(n-5,4)$ is extremal for every $n\ge13$, while $B(7,4)$ is extremal at $n=12$ by Proposition~\ref{prop:small-orders-analytic}.
\end{proof}

In particular, Theorem~\ref{thm:main} gives \(\Delta_n=\frac{n}{2}+O(1),\) and hence
\[
\frac{\Delta_n}{n}\longrightarrow\frac12 \quad (n\to\infty).
\]
Thus the largest possible separation between the adjacency spectral center and the characteristic set is asymptotically one half of the order of the tree.

\bigskip
\noindent\textit{Classification of odd-order extremal trees}
\bigskip

\begin{proof}[Proof of Theorem~\ref{thm:odd-extremal}]
We begin with necessity. Suppose that \(d(T)=d\) and \(|V(T)|=2d+11\). Choose \(s\in S(T)\) and \(c\in C(T)\) with \(d_T(s,c)=d\), and retain the notation $B,H$ from Section~\ref{sec:pathification}.

By Corollary~\ref{cor:equality},
\[
|V(B)|=2d+5,\quad H-s\cong K_{1,4},
\]
and every vertex on the \(s\)--\(c\) path has degree two in $T$. Lemma~\ref{lem:two-branch} therefore gives
\[
\rho(A(B))\le \rho_A(H-s)=\rho(A(K_{1,4}))=2.
\]

By Smith's classification of connected graphs with adjacency spectral radius at most two~\cite{Smith1970}, a tree of order at least eleven satisfying this bound must be a path, a Dynkin tree of type $D$, or an extended Dynkin tree of type \(\widetilde D\).

Write the $s$--$c$ path as \(s=x_0,x_1,\ldots,x_d=c\). Since $x_1$ is a leaf of $B$ and \(x_2,\ldots,x_d\) have degree two in $B$, the tree $B$ has a pendant path of length at least \(d\ge 3\) before reaching any branching vertex. Every leaf of an extended Dynkin tree of type \(\widetilde D\) is adjacent to a branching vertex; hence $B$ cannot be of type \(\widetilde D\). It follows that $B$ is either a path or a Dynkin tree of type $D$, with $x_1$ at the end of its long arm in the latter case.

Since \(H-s\cong K_{1,4}\), the vertex $s$ is adjacent to either the centre or a leaf of this star. Combining these two possibilities with the two possible structures of $B$, we arrive precisely at
\[
B(2d+6,4),\quad B(2d+7,3),\quad D(2d+4;2,4),\quad D(2d+5;2,3).
\]
This proves necessity.

Conversely, we verify that all four trees are extremal. We first record a consequence of Lemma~\ref{lem:outward-grafting-short}: for \(a\ge 1\) and \(\ell\ge 2\),
\begin{equation}\label{eq:branch-weight}
\beta(W(a,\ell))<\beta(W(a+1,\ell-1)).
\end{equation}
Indeed, choosing one terminal leaf to extend the stem and moving the other terminal leaves to this new endpoint, each outward grafting strictly increases the Perron value.

Consider first the two brooms. For \(B(2d+6,4)\), Lemma~\ref{lem:broom-center-short} gives
\[
S(B(2d+6,4))=\{v_{2d+5},v_{2d+6}\}.
\]
For \(B(2d+7,3)\), the unique maximum adjacency branch at \(v_{2d+5}\) points toward \(v_{2d+6}\), since the corresponding branch is \(K_{1,4}\) with spectral radius $2$, while the opposite branch is a path with spectral radius less than $2$. At \(v_{2d+6}\), the unique maximum branch points back toward \(v_{2d+5}\), because
\[
2\cos\frac{\pi}{2d+7}>\sqrt{3}=\rho(A(K_{1,3})).
\]
By Lemma~\ref{lem:adjacency-branch},
\[
S(B(2d+7,3))=\{v_{2d+5},v_{2d+6}\}.
\]

We next determine the characteristic sets. At \(v_{d+4}\) in \(B(2d+6,4)\), Lemma~\ref{lem:broom-benchmarks} gives
\[
\beta(W(d+2,4))>p_{d+5}>p_{d+4},
\]
while for \(B(2d+7,3)\), inequality~(\ref{eq:branch-weight}) gives
\[
\beta(W(d+3,3))>\beta(W(d+2,4))>p_{d+4}.
\]
Thus in both brooms the unique Perron branch at \(v_{d+4}\) points toward \(v_{d+5}\). At \(v_{d+5}\), the opposite branch is an endpoint-rooted path \(P_{d+5}\), whereas the branch toward the broom head is a nonpath rooted tree of order \(d+5\); by Lemma~\ref{lem:path-max-perron}, the former has strictly larger Perron value. Theorem~\ref{thm:perron-orientation} therefore yields
\[
C(B(2d+6,4))=C(B(2d+7,3))=\{v_{d+4},v_{d+5}\}.
\]
Consequently, both brooms have separation $d$.

Now consider \(T=D(2d+4;2,4)\), and put \(k=2d+4\). At \(v_{k-1}\), the branch toward \(v_k\) is \(K_{1,4}\) and has spectral radius $2$, whereas the opposite branch is a Dynkin tree of type $D$ and has spectral radius less than $2$. At $v_k$, the unique maximum adjacency branch points toward \(v_{k-1}\). Hence Lemma~\ref{lem:adjacency-branch} gives
\[
S(T)=\{v_{2d+3},v_{2d+4}\}.
\]
At \(v_{d+2}\), the two Perron branch values are \(\beta(W(d+2,2))\) and \(\beta(W(d+2,4))\), the second being strictly larger by Lemma~\ref{lem:rooted-subtree}. At \(v_{d+3}\), inequality~(\ref{eq:branch-weight}) gives
\[
\beta(W(d+3,2))>\beta(W(d+2,3))>\beta(W(d+1,4)).
\]
Thus the unique Perron branches at \(v_{d+2}\) and \(v_{d+3}\) point toward each other. By Theorem~\ref{thm:perron-orientation}, \(C(T)=\{v_{d+2},v_{d+3}\},\) so $d(T)=d$.

Finally, let \(T=D(2d+5;2,3)\) and put \(k=2d+5\). At \(v_{k-2}\), the branch toward \(v_{k-1}\) is \(K_{1,4}\) and has spectral radius $2$, whereas the opposite branch is a Dynkin tree of type $D$ and has spectral radius less than $2$. At \(v_{k-1}\), the branch pointing toward \(v_{k-2}\) is a Dynkin tree of type $D$ that contains $D_5$ as a subtree. Its adjacency spectral radius is therefore at least
\[
\rho(A(D_5))=\sqrt{2+\sqrt{2}}>\sqrt{3}.
\]
The branch in the opposite direction is \(K_{1,3}\), whose adjacency spectral radius is \(\sqrt{3}\). Hence the unique maximum adjacency branch at \(v_{k-1}\) points toward \(v_{k-2}\). Combining this with the comparison at \(v_{k-2}\) made above, Lemma~\ref{lem:adjacency-branch} gives
\[
S(T)=\{v_{2d+3},v_{2d+4}\}.
\]

At \(v_{d+2}\), Lemma~\ref{lem:rooted-subtree} and inequality~(\ref{eq:branch-weight}) give
\[
\beta(W(d+3,3))>\beta(W(d+2,4))>\beta(W(d+2,2)),
\]
and at \(v_{d+3}\), inequality~(\ref{eq:branch-weight}) gives
\[
\beta(W(d+3,2))>\beta(W(d+2,3)).
\]
Again the two unique Perron branches point toward each other, so Theorem~\ref{thm:perron-orientation} yields \(C(T)=\{v_{d+2},v_{d+3}\}\) and \(d(T)=d\).

All four trees are therefore extremal. They are pairwise nonisomorphic, as is seen from their degree sequences and the locations of their branching vertices. The proof is complete.
\end{proof}

\section{Open problems}\label{sec:Open problems}
Theorem~\ref{thm:odd-extremal} gives a complete classification of the extremal trees of every odd order \(n\ge 17\). The classification of extremal trees remains open for even orders \(n\ge 12\) and for the odd orders \(n=13,15\).

It would also be of interest to determine whether analogous comparisons of branch weights yield sharp separation bounds for other pairs of central sets in trees.


\begin{thebibliography}{99}

\bibitem{AbreuFritscherJustelKirkland2017}
N.~Abreu, E.~Fritscher, C.~Justel and S.~Kirkland,
On the characteristic set, centroid, and centre for a tree,
\emph{Linear Multilinear Algebra} 65 (2017) 2046--2063.
\url{https://doi.org/10.1080/03081087.2017.1304520}

\bibitem{AndradeCiardoDahl2022}
E.~Andrade, L.~Ciardo and G.~Dahl,
Perron values and classes of trees,
\emph{Linear Algebra Appl.} 639 (2022) 135--158.
\url{https://doi.org/10.1016/j.laa.2022.01.005}

\bibitem{AndradeDahl2017}
E.~Andrade and G.~Dahl,
Combinatorial Perron values of trees and bottleneck matrices,
\emph{Linear Multilinear Algebra} 65 (2017) 2387--2405.
\url{https://doi.org/10.1080/03081087.2016.1274363}

\bibitem{Bapat2014}
R.~B.~Bapat,
\emph{Graphs and Matrices},
2nd ed., Springer, London, 2014.
\url{https://doi.org/10.1007/978-1-4471-6569-9}

\bibitem{BrouwerHaemers2012}
A.~E.~Brouwer and W.~H.~Haemers,
\emph{Spectra of Graphs},
Springer, New York, 2012.
\url{https://doi.org/10.1007/978-1-4614-1939-6}

\bibitem{Ciardo2020}
L.~Ciardo,
A Fiedler center for graphs generalizing the characteristic set,
\emph{Linear Algebra Appl.} 584 (2020) 197--220.
\url{https://doi.org/10.1016/j.laa.2019.09.015}

\bibitem{CollatzSinogowitz1957}
L.~Collatz and U.~Sinogowitz,
Spektren endlicher Grafen,
\emph{Abh. Math. Sem. Univ. Hamburg} 21 (1957) 63--77.
\url{https://doi.org/10.1007/BF02941924}

\bibitem{Fiedler1975}
M.~Fiedler,
A property of eigenvectors of nonnegative symmetric matrices
and its application to graph theory,
\emph{Czechoslovak Math. J.} 25 (1975) 619--633.
\url{https://doi.org/10.21136/CMJ.1975.101357}

\bibitem{KirklandNeumannShader1996}
S.~Kirkland, M.~Neumann and B.~L.~Shader,
Characteristic vertices of weighted trees via Perron values,
\emph{Linear Multilinear Algebra} 40 (1996) 311--325.
\url{https://doi.org/10.1080/03081089608818448}

\bibitem{KumarMoharPragadaZhan2025}
H.~Kumar, B.~Mohar, S.~Pragada and H.~Zhan,
On the second largest adjacency eigenvalue of trees
with given diameter,
\emph{Linear Multilinear Algebra} 73 (2025) 2944--2972.
\url{https://doi.org/10.1080/03081087.2025.2484265}

\bibitem{Merris1987}
R.~Merris,
Characteristic vertices of trees,
\emph{Linear Multilinear Algebra} 22 (1987) 115--131.
\url{https://doi.org/10.1080/03081088708817827}

\bibitem{Neumaier1982}
A.~Neumaier,
The second largest eigenvalue of a tree,
\emph{Linear Algebra Appl.} 46 (1982) 9--25.
\url{https://doi.org/10.1016/0024-3795(82)90022-2}

\bibitem{PandeyPatra2022}
D.~Pandey and K.~L.~Patra,
Different central parts of trees and their pairwise distances,
\emph{Linear Multilinear Algebra} 70 (2022) 3790--3802.
\url{https://doi.org/10.1080/03081087.2020.1856027}

\bibitem{Patra2007}
K.~L.~Patra,
Maximizing the distance between center, centroid
and characteristic set of a tree,
\emph{Linear Multilinear Algebra} 55 (2007) 381--397.
\url{https://doi.org/10.1080/03081080701208512}

\bibitem{Smith1970}
J.~H.~Smith,
Some properties of the spectrum of a graph,
in: \emph{Combinatorial Structures and Their Applications},
Gordon and Breach, New York, 1970, pp.~403--406.

\bibitem{SmithSzekelyWangYuan2018}
H.~Smith, L.~Sz\'ekely, H.~Wang and S.~Yuan,
On different ``middle parts'' of a tree,
\emph{Electron. J. Combin.} 25 (2018) P3.17.
\url{https://doi.org/10.37236/6408}

\end{thebibliography}
\end{document}